\documentclass[12pt]{amsart}
\usepackage{}
\usepackage{amsmath}
\usepackage{amsfonts}
\usepackage{amssymb}
\usepackage[all]{xy}           %xypic macro for latex2.09

\usepackage{bbding}
\usepackage{txfonts}
\usepackage[shortlabels]{enumitem}
\usepackage{ifpdf}
\ifpdf
  \usepackage[colorlinks,final,backref=page,hyperindex]{hyperref}
\else
  \usepackage[colorlinks,final,backref=page,hyperindex,hypertex]{hyperref}
\fi
\usepackage{tikz}
\usepackage[active]{srcltx}

\makeatletter

\newtheorem{theorem}{Theorem}[section]
\newtheorem{prop}[theorem]{Proposition}
\newtheorem{defn}[theorem]{Definition}
\newtheorem{lemma}[theorem]{Lemma}
\newtheorem{coro}[theorem]{Corollary}
\newtheorem{prop-def}{Proposition-Definition}[section]

\newtheorem{exam}[theorem]{Example}

\newtheorem{assumption}[theorem]{Assumption}
\newcommand{\nc}{\newcommand}

\nc{\delete}[1]{{}}
\nc{\mmargin}[1]{}

\nc{\mlabel}[1]{\label{#1}}  % Use this to suppress names
\nc{\mcite}[1]{\cite{#1}}  % Use this to suppress names
\nc{\mref}[1]{\ref{#1}}  % Use this to suppress names
\nc{\mbibitem}[1]{\bibitem{#1}} % Use this to show number
\delete{
\nc{\mlabel}[1]{\label{#1}  % Use the next two lines to show names
{\hfill \hspace{1cm}{\bf{{\ }\hfill(#1)}}}}
\nc{\mcite}[1]{\cite{#1}{{\bf{{\ }(#1)}}}}  % Use this lines to show names
\nc{\mref}[1]{\ref{#1}{{\bf{{\ }(#1)}}}}  % Use this lines to show names
\nc{\mbibitem}[1]{\bibitem[\bf #1]{#1}} % Use this to show name
}

\nc{\difforg}{U_D}
\nc{\rbforg}{U_{RB}}
\nc{\multforg}{U_{M}}
\nc{\diffree}{{}_DF}
\nc{\difcof}{F_D}
\nc{\rbfree}{{}_{RB}F}
\nc{\rbcof}{F_{RB}}
\nc{\multfree}{{}_{M}F}
\nc{\multcof}{F_M}

\nc{\vep}{\varepsilon}
\nc{\bin}[2]{ (_{\stackrel{\scs{#1}}{\scs{#2}}})}  %binomial coeff
\nc{\binc}[2]{(\!\! \begin{array}{c} \scs{#1}\\
    \scs{#2} \end{array}\!\!)}  %binomial coeff
\nc{\bincc}[2]{  ( {\scs{#1} \atop
    \vspace{-1cm}\scs{#2}} )}  %binomial coeff
\nc{\bs}{\bar{S}}
\nc{\la}{\longrightarrow}
\nc{\ot}{\otimes}
\nc{\rar}{\rightarrow}
\nc{\dar}{\downarrow}
\nc{\dap}[1]{\downarrow \rlap{$\scriptstyle{#1}$}}
\nc{\defeq}{\stackrel{\rm def}{=}}
\nc{\dis}[1]{\displaystyle{#1}}
\nc{\dotcup}{\ \displaystyle{\bigcup^\bullet}\ }
\nc{\hcm}{\ \hat{,}\ }
\nc{\hts}{\hat{\otimes}}
\nc{\hcirc}{\hat{\circ}}
\nc{\lleft}{[}
\nc{\lright}{]}
\nc{\curlyl}{\left \{ \begin{array}{c} {} \\ {} \end{array}
    \right .  \!\!\!\!\!\!\!}
\nc{\curlyr}{ \!\!\!\!\!\!\!
    \left . \begin{array}{c} {} \\ {} \end{array}
    \right \} }
\nc{\longmid}{\left | \begin{array}{c} {} \\ {} \end{array}
    \right . \!\!\!\!\!\!\!}
\nc{\ora}[1]{\stackrel{#1}{\rar}}
\nc{\ola}[1]{\stackrel{#1}{\la}}%${\Bbb Z}$
\nc{\scs}[1]{\scriptstyle{#1}} \nc{\mrm}[1]{{\rm #1}}
\nc{\dirlim}{\displaystyle{\lim_{\longrightarrow}}\,}
\nc{\invlim}{\displaystyle{\lim_{\longleftarrow}}\,}
\nc{\dislim}[1]{\displaystyle{\lim_{#1}}} \nc{\colim}{\mrm{colim}}
\nc{\mvp}{\vspace{0.3cm}} \nc{\tk}{^{(k)}} \nc{\tp}{^\prime}
\nc{\ttp}{^{\prime\prime}} \nc{\svp}{\vspace{2cm}}
\nc{\vp}{\vspace{8cm}}
\nc{\modg}[1]{\!<\!\!{#1}\!\!>}
\nc{\intg}[1]{F_C(#1)}
\nc{\lmodg}{\!<\!\!}
\nc{\rmodg}{\!\!>\!}
\nc{\cpi}{\widehat{\Pi}}
\nc{\sha}{{\mbox{\cyr X}}}  %used to be \cyr
\nc{\ssha}{{\mbox{\cyrs X}}} %sha as product
\nc{\tsha}{{\mbox{\cyrt X}}}
\nc{\shpr}{\diamond}    %Shuffle product
\nc{\labs}{\mid\!}
\nc{\rabs}{\!\mid}

\font\cyr=wncyr10
\font\cyrs=wncyr7
\font\cyrt=wncyr5

\nc{\ann}{\mrm{ann}}
\nc{\Aut}{\mrm{Aut}}
\nc{\can}{\mrm{can}}
\nc{\Cont}{\mrm{Cont}}
\nc{\rchar}{\mrm{char}}
\nc{\cok}{\mrm{coker}}
\nc{\dtf}{{R-{\rm tf}}}
\nc{\dtor}{{R-{\rm tor}}}

\nc{\Div}{{\mrm Div}}
\nc{\End}{\mrm{End}}
\nc{\Ext}{\mrm{Ext}}
\nc{\Fil}{\mrm{Fil}}
\nc{\Fr}{\mrm{Fr}}
\nc{\Frob}{\mrm{Frob}}
\nc{\Gal}{\mrm{Gal}}
\nc{\GL}{\mrm{GL}}
\nc{\Hom}{\mrm{Hom}}
\nc{\hsr}{\mrm{H}}
\nc{\hpol}{\mrm{HP}}
\nc{\id}{\mrm{id}}
\nc{\im}{\mrm{im}}
\nc{\incl}{\mrm{incl}}
\nc{\length}{\mrm{length}}
\nc{\mchar}{\rm char}
\nc{\mpart}{\mrm{part}}
\nc{\ql}{{\QQ_\ell}}
\nc{\qp}{{\QQ_p}}
\nc{\rank}{\mrm{rank}}
\nc{\rcot}{\mrm{cot}}
\nc{\rdef}{\mrm{def}}
\nc{\rdiv}{{\rm div}}
\nc{\rtf}{{\rm tf}}
\nc{\rtor}{{\rm tor}}
\nc{\res}{\mrm{res}}
\nc{\SL}{\mrm{SL}}
\nc{\Spec}{\mrm{Spec}}
\nc{\tor}{\mrm{tor}}
\nc{\Tr}{\mrm{Tr}}
\nc{\tr}{\mrm{tr}}

\nc{\ab}{\mathbf{Ab}}
\nc{\DRB}{\mathbf{DRB}}
\nc{\OPA}{\mathbf{OPA}}
\nc{\Alg}{{\mathbf{Alg}}}
\nc{\ALG}{{\mathbf{ALG}}}
\nc{\RB}{\mathbf{RB}}
\nc{\RBA}{\mathbf{RBA}}
\nc{\bfk}{{\bf k}}
\nc{\bfone}{{\bf 1}}
\nc{\bfzero}{{\bf 0}}
\nc{\detail}{\marginpar{\bf More detail}
    \noindent{\bf Need more detail!}
    \svp}
\nc{\Diff}{\mathbf{Diff}}
\nc{\gap}{\marginpar{\bf Incomplete}\noindent{\bf Incomplete!!}
    \svp}
\nc{\FMod}{\mathbf{FMod}}
\nc{\Int}{\mathbf{Int}}
\nc{\Mon}{\mathbf{Mon}}
\nc{\Mult}{{\mathbf{Mlt}}}
\nc{\remarks}{\noindent{\bf Remarks: }}
\nc{\Rep}{\mathbf{Rep}}
\nc{\Rings}{\mathbf{Rings}}
\nc{\Sets}{\mathbf{Sets}}

\nc{\BA}{{\mathbb A}}
\nc{\CC}{{\mathbb C}}
\nc{\DD}{{\mathbb D}}
\nc{\EE}{{\mathbb E}}
\nc{\FF}{{\mathbb F}}
\nc{\GG}{{\mathbb G}}
\nc{\HH}{{\mathbb H}}
\nc{\LL}{{\mathbb L}}
\nc{\NN}{{\mathbb N}}
\nc{\PP}{{\mathbb P}}
\nc{\QQ}{{\mathbb Q}}
\nc{\RR}{{\mathbb R}}
\nc{\TT}{{\mathbb T}}
\nc{\VV}{{\mathbb V}}
\nc{\ZZ}{{\mathbb Z}}
\nc{\TP}{\widetilde{P}}
\nc{\lp}{\widehat{P}}
\nc{\lpt}{\widehat{P}^{\, \omega}}
\nc{\lqb}{\widehat{Q}}
\nc{\lqt}{\widehat{Q}^{\, \omega}}
\nc{\ld}{\hat{d}}
\nc{\ldt}{\hat{d}^{\, \omega}}
\nc{\lqst}{\hat{q}^{\, \omega}}
\nc{\lqs}{\hat{q}}
\nc{\lpp}{\overline{P'}}
\nc{\dee}{\mathrm{Deg}}
\nc{\Mpf}{\mathbf{M}}
\nc{\m}{\iota}

\nc{\cala}{{\mathcal A}}
\nc{\calc}{{\mathcal C}}
\nc{\cald}{\mathcal{D}}
\nc{\cale}{{\mathcal E}}
\nc{\calf}{{\mathcal F}}
\nc{\calg}{{\mathcal G}}
\nc{\calh}{{\mathcal H}}
\nc{\cali}{{\mathcal I}}
\nc{\call}{{\mathcal L}}
\nc{\calm}{{\mathcal M}}
\nc{\caln}{{\mathcal N}}
\nc{\calo}{{\mathcal O}}
\nc{\calp}{{\mathcal P}}
\nc{\calr}{{\mathcal R}}
\nc{\cals}{{\mathcal S}}
\nc{\calt}{{\Omega}}
\nc{\calw}{{\mathcal W}}
\nc{\calx}{{\mathcal X}}
\nc{\CA}{\mathcal{A}}

\nc{\fraka}{{\mathfrak a}}
\nc{\frakb}{\mathfrak{b}}
\nc{\frakB}{{\frak B}} \nc{\frakm}{{\frak
m}} \nc{\frakM}{{\frak M}}
\nc{\frakp}{{\frak p}}
\nc{\frakS}{{\frak S}}
\nc{\frakA}{{\frak A}} \nc{\frakx}{{\frakx}}

\nc{\Dif}{\mathbf{Dif}}
\nc{\DIF}{\mathbf{DIF}}
\nc{\ADR}{\mathbf{ADR}}
\nc{\OA}{\mathbf{OA}}
\nc{\ODA}{\mathbf{ODA}}
\nc{\ORB}{\mathbf{ORB}}
\nc{\DaRB}{\mathbf{DaRB}}
\nc{\G}{\mathbf{G}}
\nc{\C}{\mathbf{C}}
\nc{\A}{\mathbf{A}}
\nc{\B}{\mathbf{B}}
\nc{\T}{\mathbf{T}}

\begin{document}
\title[Extensions of operators, liftings of monads, and mixed distributive laws]{Extensions of operators, liftings of monads, and distributive laws}

\date{\today}

\author{Shilong Zhang}
\address{Department of Mathematics, Lanzhou University, Lanzhou, Gansu, 730000, China}
\email{2663067567@qq.com}

\author{Li Guo}
\address{
Department of Mathematics and Computer Science,
Rutgers University,
Newark, NJ 07102, USA}
\email{liguo@newark.rutgers.edu}

\author{William Keigher}
\address{
Department of Mathematics and Computer Science,
Rutgers University,
Newark, NJ 07102, USA}
\email{keigher@newark.rutgers.edu}

\begin{abstract}
In a previous study, the algebraic formulation of the First Fundamental Theorem of Calculus (FFTC) is shown to allow extensions of differential and Rota-Baxter operators on the one hand, and to give rise to categorical explanations using the ideas of liftings of monads and comonads, and mixed distributive laws on the other. Generalizing the FFTC, we consider in this paper a class of constraints between a differential operator and a Rota-Baxter operator. For a given constraint, we show that the existences of extensions of differential and Rota-Baxter operators, of liftings of monads and comonads, and of mixed distributive laws are equivalent.
\end{abstract}

\subjclass[2010]{18C15, 13N99, 16W99}

\keywords{
Rota-Baxter algebra, differential algebra, operated algebra, extension of operators, monad, comonad, distributive law}

\maketitle

\vspace{-1cm}

\tableofcontents

\vspace{-1cm}

\setcounter{section}{0}

\allowdisplaybreaks

\section{Introduction}

An {\bf operated algebra} is an associative algebra $R$ together with a linear operator on $R$. It was introduced in 1960 by Kurosh~\mcite{Ku}. A special case which had been studied much earlier is a {\bf differential algebra}~\cite{Ri}, where the linear operator $d$ satisfies the Leibniz rule
\begin{equation}
 d(xy)=d(x)y+xd(y)\ \text{ for all }\ x,\ y\in R.
\mlabel{eq:der10}
\end{equation}
More generally, for a given scalar $\lambda$, a {\bf differential operator of weight $\lambda$} satisfies
\begin{equation}
 d(xy)=d(x)y+xd(y)+\lambda d(x)d(y)\ \text{ for all }\ x,\ y\in R
\mlabel{eq:der1}
\end{equation}
and
$d(\bfone_{R})=0.$ Another example of an operated algebra is a {\bf Rota-Baxter algebra of weight $\lambda$ }~\mcite{Ba}, where the operator $P$ satisfies
\begin{equation}
 P(x)P(y)=P(P(x)y)+P(xP(y))+\lambda P(xy)\ \text{ for all }\ x,\ y\in R.
\mlabel{eq:bax1}
\end{equation}

Both differential algebra and Rota-Baxter algebra arose as the algebraic abstractions of differential calculus and integral calculus, respectively. Their extensive studies have established the subjects as important areas of mathematics with broad applications in mathematics and physics \mcite{Bai,Ca,CK,Gub,GK1,GZ,Kol,RR,Ro,Wu2}. Bringing together the notions of a differential algebra and a Rota-Baxter algebra results in the concept of a differential Rota-Baxter algebra, where the differential operator and Rota-Baxter operator are paired through an abstraction of the First Fundamental Theorem of Calculus (FFTC). See~\mcite{GGR,GRR,RR} for a variation, called an integro-differential algebra.

As it turned out, this coupling of algebraic operators with analytic origins has important categorical implications.
Indeed in~\mcite{ZGK}, we gave a mixed distributive law to differential Rota-Baxter algebras.
To be precise, for any algebra $R$, let $(R^\NN, \partial_R)$ be the cofree differential algebra on $R$, where $R^\NN$ denotes the Hurwitz series algebra. Let $(\sha(R), P_R)$ be the free Rota-Baxter algebra on $R$, where $\sha(R)$ is constructed by the mixable shuffle product. In~\mcite{ZGK}, a differential operator on $R$ is uniquely extended to one on $\sha(R)$, enriching $\sha(R)$ to a differential Rota-Baxter algebra giving the free differential Rota-Baxter algebra. Similarly, a Rota-Baxter operator on $R$ is uniquely extended to one on $R^\NN$, again enriching $R^\NN$ to a differential Rota-Baxter algebra, yielding the cofree differential Rota-Baxter algebra. These extensions of operators further give the liftings of (co)monads\footnote{We use the convention `(co)word' to indicate the use of the notion of `word' and its dual `coword'.  This could apply to monad and extension, etc.}, which in turn give a mixed distributive law.
These results suggest close connections between extensions of differential and Rota-Baxter operators, liftings of (co)monads, and mixed distributive laws.

In order to better understand the interrelationships among these properties, we should work in a broader context in which such properties can be distinguished. This is the motivation of this follow-up study. The identity in the FFTC is viewed as an instance of a polynomial identity in two noncommutative variables symbolizing the differential operator and Rota-Baxter operator, regarded as a more general constraint between the two operators exemplified by the FFTC. We explore categorical consequences of these constraints, including extensions of operators to free Rota-Baxter algebras and cofree differential algebras, liftings of (co)monads on a richer category, and existence of mixed distributive laws.

To get some sense on how things should work in general, we consider a class of constraints which is special enough to be manageable yet broad enough to include the commonly known instances and to reveal the dependence of these categorical properties on the constraints. Thus we introduce in Section~\mref{sec:ext} a class $\calt$ of polynomials in two noncommutative variables $x$ and $y$.
Each element $\omega:=\omega(x, y)$ in $\calt$ is regarded as a coupling of a differential operator $d$ and a linear operator $Q$ given by a formal identity $\omega(d, Q)=0$. Then the triple $(R, d, Q)$ will be called a type $\omega$ operated differential algebra. Similarly, a type $\omega$ operated Rota-Baxter algebra $(R, q, P)$ consists of a linear operator $q$ satisfying $q(\bfone_R)=0$ and a Rota-Baxter operator $P$ with the identity $\omega(q, P)=0$. As a special case, a type $\omega$ differential Rota-Baxter algebra $(R, d, P)$ satisfies the identity $\omega(d, P)=0$ between the differential operator $d$ and Rota-Baxter operator $P$.
The FFTC in a differential Rota-Baxter algebra corresponds to the case of $\omega(x, y)=xy-1$.

Let $\OA$ be the category of operated algebras, and $\OA_0$ be the subcategory of $\OA$ which consists of all operated algebras $(R, q)$ with the property $q(\bfone_R)=0$. Let $\DIF$ and $\RBA$ be the categories of differential algebras and Rota-Baxter algebras, respectively. Then $\DIF$ and $\RBA$ are subcategories of $\OA_0$ and $\OA$, respectively.
Furthermore, let $\ODA_\omega, \ORB_\omega$ and $\DRB_\omega$ denote the categories of type $\omega$ operated differential algebras, type $\omega$ operated Rota-Baxter algebras and type $\omega$ differential Rota-Baxter algebras, respectively.
We provide a canonical way to extend a linear operator $Q$ on an algebra $R$ to one on the differential algebra $(R^\NN, \partial_R)$, giving rise to a functor
$G_\omega: \OA \to \ODA_\omega$. We likewise provide a canonical way to extend a linear operator $q$ on an algebra $R$ with $q(\bfone_R)=0$ to one on the Rota-Baxter algebra $(\sha(R), P_R)$, giving rise to a functor
$F_\omega: \OA_0 \to \ORB_\omega$. It is natural to ask whether the restriction of $G_\omega$ to the subcategory $\RBA$ of $\OA$ gives a functor $\RBA \to \DRB_\omega$. Likewise for $F_\omega$, as indicated in the following diagram.

\begin{equation*}
\xymatrix{
&\ODA_\omega&\DRB_\omega\ar@{_{(}->}[l]\ar@{^{(}->}[r]&\ORB_\omega&\\
&&&&\\
\OA\ar^{G_\omega}[ruu]&\, \RBA \ar@{_{(}->}[l]\ar^{G_\omega}[ruu]&&\DIF \ar@{^{(}->}[r]\ar_{F_\omega}[luu] &\OA_0\ar_{F_\omega}[luu]
}
\mlabel{eq:cat}
\end{equation*}
We show in Theorem~\mref{thm:main} that this natural expectation on restrictions of functors has equivalent statements in terms of liftings of (co)monads and corresponding mixed distributive laws.

\smallskip

Throughout the paper, we fix a commutative ring $\bfk$ with identity and an element $\lambda \in \bfk$. Unless otherwise noted, all algebras we consider will be commutative
$\bfk$-algebras with identity, and all operators are also $\bfk$-linear. All homomorphisms of algebras will be $\bfk$-algebra homomorphisms that preserve the identity, and all homomorphisms of operated algebras will be homomorphisms of algebras which commute with operators. Thus references to $\bfk$ will be suppressed unless a specific $\bfk$ is emphasized or a reminder is needed.
We write $\NN$ for the additive monoid of
natural numbers $\{0,1,2,\ldots\}$ and
$\NN_+=\{ n\in \NN\mid n>0\}$ for the positive integers.
In this paper, we use the categorical notations as in ~\cite{Ma}.

\section{Extensions of operators to cofree differential algebras and free Rota-Baxter algebras}
\mlabel{sec:ext}
After providing background on algebras with one operator, including Rota-Baxter algebras, differential algebras and their (co)free objects, we introduce the key concepts on algebras with two operators, including type $\omega$ operated differential algebras and type $\omega$ operated Rota-Baxter algebras. Suitable (co)extensions of operators to these algebras are studied.

\subsection{Free Rota-Baxter algebras and cofree differential algebras}
\mlabel{ss:trb}

We begin with some background on Rota-Baxter algebras defined in Eq.~(\mref{eq:bax1}). Additional details can be found in~\mcite{GK1,Gub}.
We first give a partial sum example of a Rota-Baxter algebra.

\begin{exam}$($\cite{Ba,Ro}$)$
{\rm
Let $R$ be the set of sequences $(a_n)_{n\in\NN_+}$ with values in $\bfk$. Then $R$ is an algebra with termwise addition, product and scalar product. Define an operator $P:R\rar R$ by
$$P(\mathfrak{a}):=(a_1,a_1+a_2,\cdots,\sum\limits_{1\leq i\leq n}a_i,\cdots)\quad \text{for each}\ \mathfrak{a}:=(a_1,a_2,\cdots, a_n,\cdots)\in R.$$
Then $P$ is a Rota-Baxter operator of weight $-1$. That is, for any $\mathfrak{a}, \mathfrak{b}\in R$,
$$P(\mathfrak{a})P(\mathfrak{b})=P(P(\mathfrak{a})\mathfrak{b})
+P(\mathfrak{a}P(\mathfrak{b}))-P(\mathfrak{a}\mathfrak{b}).$$
Further, the operator $Q:R\rar R$ defined by
$$Q(\mathfrak{a}):=(0,a_1,a_1+a_2,\cdots,\sum\limits_{i< n}a_i,\cdots)$$
is a Rota-Baxter operator of weight $1$.
}\mlabel{ex:parsum}
\end{exam}

Constructions of free commutative Rota-Baxter algebras on sets were first obtained by Rota and Cartier in~\cite{Ca,Ro}.
We recall from~\cite{GK1} the construction of the free commutative Rota-Baxter algebra $\sha(A)$ of weight $\lambda$ on a commutative
algebra $A$ with identity $\bfone_A$. As a module, we have
$$\sha(A) = \bigoplus\limits_{i\in\NN_+}A^{\otimes i}=A\oplus (A\otimes A)\oplus (A\otimes A\otimes A)\oplus\cdots$$
where the tensors are defined over $\bfk$.
The product for this free Rota-Baxter algebra on $A$ is constructed in terms of a generalization of the shuffle product, called the mixable shuffle product which in its recursive form is a natural generalization of the quasi-shuffle product~\cite{Ho}. We will recall the general construction of the mixable shuffle product as follows.

Let $\mathfrak{a}=a_0\otimes \cdots \otimes a_m\in A^{\otimes(m+1)}$ and $\mathfrak{b}=b_0\otimes \cdots \otimes b_n\in A^{\otimes(n+1)}$. If $mn=0$, define

\begin{equation}
\mathfrak{a}\diamond \mathfrak{b}=
\begin{cases}
(a_0b_0)\otimes b_1 \otimes\cdots \otimes b_n, & m=0, n>0, \\
(a_0b_0)\otimes a_1 \otimes\cdots \otimes a_m, & m>0, n=0, \\
a_0b_0,                                        & m=n=0. \\
\end{cases}
\mlabel{eq:mshprod1}
\end{equation}
If $m>0$ and $n>0$, then $\mathfrak{a}\diamond \mathfrak{b}$ is defined inductively on $m+n$ by
\begin{eqnarray}
&&(a_0b_0)\otimes\Big((a_1\otimes\cdots\otimes a_m)\diamond (\bfone_A\otimes b_1\otimes \cdots \otimes b_n)
                 +(\bfone_A\otimes a_1\otimes \cdots\otimes a_m)\diamond (b_1\otimes \cdots \otimes b_n)\notag\\
&&                 +\lambda (a_1\otimes\cdots\otimes a_m)\diamond (b_1\otimes \cdots \otimes b_n)\Big).
\mlabel{eq:mshprod2}
\end{eqnarray}
Extending by additivity, $\shpr$ gives a $\bfk$-bilinear map

\[ \shpr: \sha(A) \times \sha(A) \rar \sha(A). \]
Since the product $\shpr$ restricts to the product on $A$, we will usually suppress the symbol $\shpr$ and simply denote $x y$ for $x\shpr y$ in $\sha(A)$.

Define an operator $P_A$ on
$\sha(A)$ by assigning
\[ P_A( x_0\otimes x_1\otimes \ldots \otimes x_n):
=\bfone_A\otimes x_0\otimes x_1\otimes \ldots\otimes x_n \]
for all
$x_0\otimes x_1\otimes \ldots\otimes x_n\in A^{\otimes (n+1)}$
and extending by additivity. We note the fact of the mixable shuffle product on $\sha(A)$: for any $u_0\ot u' \in A^{\ot(n+1)}$ with $u'\in A^{\ot n}$,
$u_0\ot u'=u_0 P_A(u').$

\begin{theorem}$($\cite[Theorem~4.1]{GK1}$)$
The module $\sha(A)$, with the mixable shuffle product, the operator $P_A$ and the natural embedding
$j_A:A\rightarrow \sha(A)$,
is a free Rota-Baxter algebra of weight $\lambda$ on $A$.
More precisely,
for any Rota-Baxter algebra $(R,P)$ of weight $\lambda$ and any
algebra homomorphism
$\varphi:A\rar R$, there exists
a unique Rota-Baxter algebra homomorphism
$\tilde{\varphi}:(\sha(A),P_A)\rar (R,P)$ such that
$\varphi = \tilde{\varphi} j_A$.
\mlabel{thm:shua}
\end{theorem}

\smallskip

We next review some background on differential algebras with weights, defined in Eq.~(\mref{eq:der1}), and refer the reader to~\mcite{GK3} for details.
Note that a differential operator of weight $0$ is just a derivation
in the usual sense~\cite{Kol}.

Recall now an example that motivates the definition of a differential operator~\mcite{GK3}. Let $\RR$ denote the real number field, and let
$\lambda \in \RR$ with
$\lambda \neq 0$.  Let $A$ denote the $\RR$-algebra of
$\RR$-valued analytic functions on $\RR$, and consider the usual
`difference quotient' operator $d_\lambda$ on $A$ defined by $$(d_\lambda(f))(x) := \frac{f(x+\lambda) - f(x)}{\lambda}\quad\text{for all}\  f\in A, x\in \RR.$$
Then $d_\lambda$ is a differential operator of weight $\lambda$ on $A$.

We next recall the concept and basic properties of the algebra of $\lambda$-Hurwitz series~\cite{GK3} as a generalization of the ring of Hurwitz series~\cite{Ke}.
For any algebra $A$, let $A^{\NN}$ denote the $\bfk$-module of all functions $f:\NN \rightarrow A$. There is a one-to-one correspondence between elements $f \in A^{\NN}$ and sequences $(f_n)=(f_0,f_1,\cdots)$ with $f_n \in A$ given by $f_n := f(n)$ for all $n \in \NN$.
The {\bf $\lambda$-Hurwitz product}~\cite[\S~2.3]{GK3} on $A^{\NN}$ is given by
\begin{equation}
(fg)_n = \sum_{k=0}^{n}\sum_{j=0}^{n-k} {n\choose k} {n-k\choose j}
\lambda^{k}f_{n-j}g_{k+j}\quad \text{for all}\  f, g\in R^\NN.
\mlabel{eq:hurprod}
\end{equation}
In particular, if $\lambda =0$, then Eq.~(\mref{eq:hurprod}) becomes
\begin{equation}
(fg)_n =\sum_{j=0}^{n} {n\choose j} f_{n-j}g_{j}\quad \text{for all}\  f, g\in R^\NN.
\mlabel{eq:hurprod0}
\end{equation}

As in~\cite{GK3}, we call $A^{\NN}$ the {\bf algebra of $\lambda$-Hurwitz series} over $A$.
Further, the operator
\begin{equation*}
\partial_A:A^{\NN} \rightarrow A^{\NN}, \quad \partial_A(f)_n = f_{n+1}\quad\text{ for all }\  f \in A^{\NN}, n\in \NN,
\mlabel{eq:partial}
\end{equation*}
is a differential operator of weight $\lambda$ on $A^{\NN}$ and then $(A^{\NN},\partial_A)$ is a  differential algebra of weight $\lambda$.
This property gives a recursive formula for $(fg)_n$:
\begin{equation}
(fg)_{n+1}=(\partial_A(fg))_n =(\partial_A(f)g)_n+(f\partial_A(g))_n+
(\lambda\partial_A(f)\partial_A(g))_n \quad\text{ for all }\  f, g \in A^{\NN}, n\in \NN.
\mlabel{eq:recHur}
\end{equation}

\begin{prop}$($\cite[Proposition~2.8]{GK3}$)$
For any algebra $A$, the differential algebra $(A^{\NN},\partial_A)$, together with the algebra homomorphism
$$\varepsilon_{A}: A^{\NN}\rightarrow A, \quad \varepsilon_A(f): = f_0\quad\text{for all}\  f \in A^{\NN},
$$
is a cofree differential algebra of weight $\lambda$ on the algebra $A$.
More precisely,
for any differential algebra $(R,d)$ of weight $\lambda$ and any
algebra homomorphism
$\varphi:R\rar A$, there exists
a unique differential algebra homomorphism
$\tilde{\varphi}:(R,d)\to (A^\NN,\partial_A)$ such that
$\varphi =\vep_A \tilde{\varphi}$.
\mlabel{prop:cofree}
\end{prop}

\subsection{(Co)extensions of operators}
\mlabel{subsec:coex}

For an operator on an algebra, we construct the coextension of the operator to the cofree differential algebra generated by this algebra. Further, assuming that the value of the operator on the identity of the algebra is zero, we also construct the extension of the operator to the free Rota-Baxter algebra generated by the algebra. For this purpose, we introduce a class of variations of operated algebras by enriching them with a differential operator or Rota-Baxter operator.

Let $\bfk\langle x, y\rangle$ be the free $\bfk$-algebra of polynomials in two noncommutative variables $x$ and $y$.
Consider the subset of $\bfk\langle x,y\rangle$:
\begin{equation}
\calt:=xy+\bfk[x]+y\bfk[x] = \{
xy-(\phi(x) +y\psi(x))\,|\, \phi,\psi\in\bfk[x]\}.
\mlabel{eq:T}
\end{equation}
Let $q$ and $Q$ be two operators on an algebra $R$. For each $\omega:=\omega(x, y)=xy-(\phi(x) +y\psi(x))\in\calt$, we regard $\omega(q, Q)=0$ as a relation between $q$ and $Q$ which describes how the operators $q$ and $Q$ interact with each other. For example, when $\omega=xy-1$, $\omega(d, P)=dP-\id_R =0$ amounts to the relation between the operators $d$ and $P$ in a differential Rota-Baxter algebra $(R, d, P)$~\cite{GK3} arising from the First Fundamental Theorem of Calculus.

Recall from the introduction that an operated algebra is an algebra $R$ with a linear operator $Q$ on $R$, thus denoted as a pair $(R, Q)$.
\begin{defn}
For a given $\omega\in\calt$ and $\lambda\in \bfk$, we say that the triple $(R,d,Q)$ is a
{\bf type $\omega$ operated differential algebra of weight $\lambda$} if
\begin{enumerate}
\item $(R,d)$ is a differential algebra of weight $\lambda$,
\item $(R,Q)$ is an operated algebra, and
\item $\omega(d, Q)=0$, that is,
   \begin{equation}
   dQ=\phi(d)+Q\psi(d).
   \mlabel{eq:dP}
   \end{equation}
\end{enumerate}
\mlabel{de:oda}
\end{defn}
As noted before, every $f\in R^\NN$ is identified with a sequence $(f_n)$ of elements in $R$. Likewise, there is a one-to-one correspondence between operators $\calp$ on $R^\NN$ and sequences $(\calp_n)$ of linear maps where, for each $n\in\NN$, $\calp_n:R^\NN\rightarrow R$ is given by
\begin{equation*}
\calp_n(f):=\calp(f)_n\quad\text{for all}\  f\in R^\NN.
\mlabel{eq:coexten}
\end{equation*}
For any operators $\mathcal{Q}, \mathcal{J}$ on $R^\NN$, and each $f\in R^\NN, n\in\NN$, we obtain
\begin{equation}
(\partial_R\mathcal{Q})_n (f)=(\partial_R(\mathcal{Q}(f)))_n=\mathcal{Q}_{n+1}(f),\quad (\mathcal{Q}\mathcal{J})_n(f)=(\mathcal{Q}(\mathcal{J}(f)))_n =\mathcal{Q}_n(\mathcal{J}(f)).
\mlabel{eq:pf0}
\end{equation}
For the remainder of the paper, we prefer to use $\calp_n(f)$ in place of $\calp(f)_n$.

We first consider coextensions of operators.
\begin{defn}
For a given operator $Q:R\to R$, we call an operator $\lqb:R^\NN\to R^\NN$ a {\bf coextension of $Q$ to $R^\NN$} if for all $f\in R^\NN$, we have $\lqb_0(f)=Q(f_0)$.
\end{defn}
We now establish the existence and uniqueness of a coextension.
\begin{prop}
Let $Q$ be an operator on an algebra $R$. For a given $\omega=xy-(\phi(x)+y\psi(x))\in \calt$ with $\phi, \psi\in \bfk[x]$, $Q$ has a unique coextension $\lqt:(R^\NN, \partial_R)\to (R^\NN, \partial_R)$ such that $\omega(\partial_R, \lqt) = 0$, that is:
\begin{equation}
\partial_R \lqt=\phi(\partial_R) +\lqt\psi(\partial_R).
\mlabel{eq:lpproperty}
\end{equation}
Thus the triple $(R^\NN, \partial_R, \lqt)$ is a type $\omega$ operated differential algebra.
\mlabel{prop:extrb}
\end{prop}
\begin{proof}
Let $\calp$ be an operator on $R^\NN$, giving by the sequence $(\calp_n)$ through the above one-to-one correspondence. Then applying Eq.~(\mref{eq:pf0}), the equation $\partial_R \calp=\phi(\partial_R)+\calp \psi(\partial_R)$ is equivalent to
\begin{equation}
\calp_{n+1}=(\partial_R \calp)_n=\phi(\partial_R)_n +\calp_n\psi(\partial_R)\quad\text{for all}\  n\in\NN.
\mlabel{eq:lptn}
\end{equation}
Thus a coextension $\lqb$ of $Q$ satisfying Eq.~(\mref{eq:lpproperty}) is equivalent to a solution $(\lqb_n)$ of the recursion in Eq.~(\mref{eq:lptn}) with $\calp=\lqb$ and with the initial condition
\begin{equation}
\lqb_0(f)=Q(f_0)\quad\text{for all}\  f\in R^\NN.
\mlabel{eq:lpt0}
\end{equation}
Then the proposition follows since this recursion has a unique solution.
\end{proof}

Next, we consider extensions of operators.
\begin{defn}
For a given $\omega\in\calt$ and $\lambda\in \bfk$, we say that the triple $(R,q,P)$ is a
{\bf type $\omega$ operated Rota-Baxter algebra of weight $\lambda$} if
\begin{enumerate}
\item $(R,q)$ is an operated algebra with the property $q(\bfone_R)=0$,
\item $(R,P)$ is a Rota-Baxter algebra of weight $\lambda$, and
\item $\omega(q, P)=0$, that is,
   \begin{equation*}
   qP=\phi(q)+P\psi(q).
   \mlabel{eq:qP}
   \end{equation*}
\end{enumerate}
\end{defn}
\begin{defn}
For a given operator $q:R\to R$ satisfying $q(\bfone_R)=0$, we call an operator $\lqs:\sha(R)\to \sha(R)$ an {\bf extension} of $q$ to $\sha(R)$ if $\lqs|_R=q$.
\end{defn}
By the definition of the mixable shuffle product on $\sha(R)$, we obtain $\bfone_{\ssha(R)}=\bfone_R$. Then an extension $\lqs$ of $q$ to $\sha(R)$ satisfies $\lqs(\bfone_{\ssha(R)})=q(\bfone_R)=0$.

\begin{prop}
Let $(R, q)$ be an operated algebra where the operator $q$ satisfies $q(\bfone_R)=0$. For a given $\omega=xy-(\phi(x)+y\psi(x))\in \calt$ with $\phi, \psi\in\bfk[x]$, $q$ has a unique extension
$$\lqst: (\sha(R), P_R)\rightarrow (\sha(R), P_R)$$
with the following property: for $u=u_0\ot u'\in R^{\ot (n+1)}$ with $u'\in R^{\ot n}$,
\begin{equation}
\lqst(u)=q(u_0)\ot u'+(u_0+\lambda q(u_0))(\phi(\lqst)+ P_R\psi(\lqst))(u')
\mlabel{eq:1leibdform}
\end{equation}
and
\begin{equation}
\lqst(\oplus_{i=1}^{n}R^{\ot i})\subseteq\oplus_{i=1}^{n}R^{\ot i}\quad\text{for each}\  n\in\NN_+.
\mlabel{eq:qcon}
\end{equation}
Therefore,
\begin{equation}
\lqst P_R=\phi(\lqst)+ P_R\psi(\lqst)
\mlabel{eq:qqcon}
\end{equation}
and then the triple
$(\sha(R), \lqst, P_R)$ is a type $\omega$ operated Rota-Baxter algebra.
\mlabel{prop:extd}
\end{prop}
\begin{proof}
Since $\sha(R)$ is the direct limit (via taking union) of its submodules $\oplus_{i=1}^n R^{\ot i}, n\in \NN_+$, by~\cite[Proposition~5.26]{Rot},
there is a one-to-one correspondence between operators $\mathcal{D}:\sha(R)\to \sha(R)$ and compatible sequences $(\mathcal{D}_n)$ of linear maps $\mathcal{D}_n:\oplus_{i=1}^{n}R^{\ot i}\to \sha(R)$ where the compatible condition means $\mathcal{D}_{n}=\mathcal{D}_{n+1}|_{\oplus_{i=1}^{n}R^{\ot i}}$ for each $n\in\NN_+$.
Thus we just need to show by induction on $n\in\NN_+$ that there is a unique compatible sequence $(\lqst_n)$ with $\lqst_n:\oplus_{i=1}^n R^{\ot i} \rightarrow \sha(R)$ extending $q$ and satisfying Eqs.~(\ref{eq:1leibdform}), (\ref{eq:qcon}) when $\lqst$ is replaced by $\lqst_n$.

As $\lqst:\sha(R)\to \sha(R)$ is an extension of $q$, we have $\lqst_1=q$ and $\lqst_1(R)\subseteq R$, verifying the case when $n=1$.

For a given $k\in\NN_+$, assume that the required $\lqst_k$ on $\oplus_{i=1}^{k}R^{\ot i}$ has been defined. Define $\lqst_{k+1}$ on $\oplus_{i=1}^{k+1} R^{\ot i}$ by first taking $\lqst_{k+1}|_{\oplus_{i=1}^kR^{\ot i}}=\lqst_k$. Further consider $v\in R^{\ot (k+1)}$ and write $v=v_0\ot v'$ with $v'\in R^{\ot k}$. By the induction hypothesis, $(\phi(\lqst_{k+1}))(v')=(\phi(\lqst_k))(v')$ and $(\psi(\lqst_{k+1}))(v')=(\psi(\lqst_k))(v')$ are also defined and satisfy
\begin{equation*}
(\phi(\lqst_{k+1}))(v')\subseteq \oplus_{i=1}^k R^{\ot i},\quad (\psi(\lqst_{k+1}))(v')\subseteq \oplus_{i=1}^k R^{\ot i}.
\mlabel{eq:phipsi}
\end{equation*}
Then we can uniquely define
$$\lqst_{k+1}(v):=q(v_0)\ot v'+(v_0+\lambda q(v_0))(\phi(\lqst_{k+1})+ P_R\psi(\lqst_{k+1}))(v').$$
Then $\lqst_{k+1}(v) \in \oplus_{i=1}^{k+1} R^{\ot i}$.
This completes the construction of the desired $\lqst_{k+1}$ and the induction.

By Eq.~(\ref{eq:1leibdform}), we obtain
$$(\lqst P_R)(u)=\lqst(\bfone_R\ot u)=q(\bfone_R)\ot u+(\bfone_R+\lambda q(\bfone_R))(\phi(\lqst)+ P_R\psi(\lqst))(u)=(\phi(\lqst)+ P_R\psi(\lqst))(u).$$ Thus Eq.~(\ref{eq:qqcon}) holds.
\end{proof}

\section{Liftings of comonads and monads, and mixed distributive laws}
\mlabel{sec:lift}
In this section, we characterize (co)extensions of operators in terms of liftings of (co)monads, as well as mixed distributive laws, summarized in the result of this paper, Theorem~\mref{thm:main}.

\subsection{The monad giving Rota-Baxter algebras and comonad giving differential algebras}
We first recall from~\mcite{ZGK} the monad giving Rota-Baxter algebras.

We let $\ALG$ denote the category of commutative algebras, and let $\RBA_{\lambda}$, or simply $\RBA$, denote the category
of commutative Rota-Baxter algebras of weight $\lambda$.

Let $U: \RBA \rar \ALG$
denote the forgetful functor by forgetting the Rota-Baxter operator.
Let $F:\ALG \rar \RBA$ denote the functor given on
objects $A$ in $\ALG$ by $F(A) = (\sha(A),P_A)$ and on morphisms
$\varphi:A \rar B$ in $\ALG$ by $$F(\varphi)\left(\sum^{k}_{i=1}a_{i0}\otimes a_{i1}\otimes\cdots \otimes a_{in_{i}}\right)=\sum^{k}_{i=1}\varphi(a_{i0})\otimes \varphi(a_{i1})\otimes\cdots \otimes \varphi(a_{in_{i}})$$
for any $\sum\limits^{k}_{i=1}a_{i0}\otimes a_{i1}\otimes\cdots \otimes a_{in_{i}}\in\sha(A)$.

Define a natural transformation
$$\eta: \id_{\ALG} \rightarrow U F $$
with
$$\eta_{R}:R\rightarrow (U F)(R) = \sha(R)\quad\text{for any}\  R\in\ALG$$
to be just the natural embedding $j_R: R \rar \sha(R)$.
There is also a natural transformation
$$\varepsilon: F U \rightarrow \id_{\RBA}$$
with
$$\varepsilon_{(R,P)}: (F  U)(R,P) = (\sha(R),P_R) \rightarrow (R,P)\quad \text{ for each }\, (R,P) \in \RBA,$$
defined by
\begin{equation}
\varepsilon_{(R,P)}\left(\sum^{k}_{i=1}u_{i0}\otimes u_{i1}\otimes\cdots \otimes u_{in_{i}}\right)=
\sum^{k}_{i=1}u_{i0}P(u_{i1}P(\cdots P(u_{in_{i}})\cdots))
\mlabel{eq:vare0}
\end{equation}
for any $\sum\limits_{i=1}^{k}u_{i0}\otimes u_{i1}\otimes\cdots \otimes u_{in_{i}} \in \sha(R).$

From a general principle of category theory~\mcite{Ma,ZGK}, equivalent to Theorem~\ref{thm:shua}, we have

\begin{coro}
The functor $F:\ALG \rightarrow \RBA$ defined above is the left adjoint of the forgetful functor
$U:\RBA \rightarrow \ALG$. In other words, there is an adjunction $\langle F, U, \eta, \varepsilon \rangle:\ALG\rightharpoonup\RBA$.
\mlabel{co:adjrba}
\end{coro}

The adjunction $\langle F, U, \eta, \varepsilon \rangle:\ALG\rightharpoonup\RBA$ gives rise to a monad $\mathbf{T}=\mathbf{T}_{\RBA}=\langle T,\eta,\mu\rangle$ on $\ALG$, where $T$ is the functor $T:=U F:\ALG \rightarrow \ALG$ and $\mu$ is the natural transformation defined by $\mu:=U\varepsilon F: TT\rightarrow T.$

From~\cite{Ma}, the monad $\mathbf{T}$ induces a category
of $\mathbf{T}$-algebras, denoted by $\ALG^{\mathbf{T}}$. The objects in $\ALG^{\mathbf{T}}$ are pairs $\langle A,h\rangle$, where $A$ is in $\ALG$ and
$h:\sha(A)\rightarrow A$ is an algebra homomorphism satisfying
\begin{equation*}
h\eta_A=\id_A,\quad\quad h T(h)=h \mu_{A}.
\mlabel{eq:Tcate}
\end{equation*}
A morphism $\phi:\langle R,f\rangle\rightarrow\langle S,g\rangle$ in $\ALG^{\mathbf{T}}$ is an algebra homomorphism $\phi:R\rightarrow S$ such that $g T(\phi)=\phi  f$.

Further, the monad $\mathbf{T}$ gives rise to an adjunction
$$\langle F^{\mathbf{T}}, U^{\mathbf{T}},\eta^{\mathbf{T}} , \varepsilon^{\mathbf{T}} \rangle:\ALG
\rightharpoonup \ALG^{\mathbf{T}},$$
where the functor
$$F^{\mathbf{T}} : \ALG \rightarrow \ALG^{\mathbf{T}}$$
is defined on objects $A$ in $\ALG$
by $F^{\mathbf{T}}(A) = \langle \sha(A),\mu_{A}\rangle$.
The functor
$$U^{\mathbf{T}}:\ALG^{\mathbf{T}}\rightarrow \ALG$$
is given on objects
$\langle A,h\rangle$ in $\ALG^{\mathbf{T}}$ by $U^{\mathbf{T}}\langle A,h\rangle = A$.
The natural transformations $\eta^{\mathbf{T}}$ and $\varepsilon^{\mathbf{T}}$
 are defined similarly as $\eta$ and $\varepsilon$, respectively.
Then there is a uniquely defined comparison functor
$$K:\RBA \rightarrow \ALG^{\mathbf{T}},\quad K(R, P)=\langle R,U(\varepsilon_{(R,P)})\rangle \quad\text{for any}\  (R, P)\in\RBA$$ such that $K F=F^{\mathbf{T}}$
and $U^{\mathbf{T}} K=U$.

\begin{coro}$($\cite[Corollary~2.6]{ZGK}$)$
The comparison functor $K:\RBA \rightarrow \ALG^{\mathbf{T}}$ is
an isomorphism, that is, $\RBA$ is monadic over $\ALG$.
\mlabel{co:rbmon}
\end{coro}

Next, we recall also from~\mcite{ZGK} the comonad giving differential algebras.

Let $\varphi:A \rightarrow B$ be an algebra homomorphism. Then the map
\begin{equation}
\varphi^{\NN}: A^{\NN} \rightarrow B^{\NN},\quad \varphi^{\NN}_n(f) := \varphi(f_n) \quad\text{ for all }\  f\in A^{\NN}, n\in\NN,
\mlabel{eq:difhomo}
\end{equation}
is a differential algebra homomorphism from $(A^{\NN},\partial_A)$ to $(B^{\NN},\partial_B)$.

Let $\DIF$ denote the category
of differential algebras of weight $\lambda$. Let $V:\DIF \rightarrow \ALG$ denote the forgetful functor.
We also have a functor
$G:\ALG \rightarrow \DIF$ given on objects
$A$ in $\ALG$ by $G(A): = (A^{\NN}, \partial_A)$ and on morphisms
$\varphi:A \rightarrow B$ in $\ALG$ by $G(\varphi):=\varphi^{\NN}$ as defined in Eq.~(\mref{eq:difhomo}).

There is a natural
transformation $\eta: \id_{\DIF} \rightarrow G V$
by
\begin{equation}
\eta_{(R,d)}:(R,d) \rightarrow (G V)(R,d)=(R^{\NN}, \partial_R),\quad (\eta_{(R,d)}(x))_n:= d^n(x)
\mlabel{eq:eta1}
\end{equation}
for any $(R,d) \in \DIF$ and all $x\in R$, $n\in \NN$.
There is also a natural transformation
$\varepsilon: V G \rightarrow \id_{\ALG}$ by
\begin{equation}
\varepsilon_A: (V G)(A) = A^{\NN} \rightarrow A,\quad \varepsilon_A(f): = f_0\quad\text{for any}\  A\in \ALG, f \in A^{\NN}.
\mlabel{eq:vare1}
\end{equation}

As an equivalent statement of Proposition~\ref{prop:cofree}, we have
\begin{coro}
The functor $G:\ALG \rightarrow \DIF$ defined above is the right adjoint of the forgetful functor $V:\DIF \rightarrow \ALG$. In other words, there is an adjunction
$\langle V,G,\eta,\varepsilon\rangle:\DIF\rightharpoonup\ALG.$
\mlabel{coro:cofree}
\end{coro}

Corresponding to the adjunction $\langle V,G,\eta,\varepsilon\rangle:\DIF\rightharpoonup\ALG$,
there is a comonad $\C = \langle C,\varepsilon,\delta \rangle$ on $\ALG$, where $C$ is the functor
$C := VG:\ALG \rightarrow \ALG$
and $\delta$ is the natural transformation from $C$ to $C C$ defined by $\delta := V\eta G$.

The comonad $\C$ induces a category
of $\C$-coalgebras, denoted by $\ALG_{\C}$.  The objects in
$\ALG_{\C}$ are pairs $\langle A,f\rangle$, where $A$ is in $\ALG$ and
$f:A \rightarrow A^{\NN}$ is a homomorphism in $\ALG$ satisfying the properties
$$\varepsilon_A  f= \id_A, \quad \delta_A  f= f^{\NN}  f.$$
A morphism $\varphi:\langle A,f\rangle\rightarrow\langle B,g\rangle$ in $\ALG_{\C}$ is an algebra homomorphism $\varphi: A \rightarrow B$ such that $g  \varphi = \varphi^{\NN}  f$.

The comonad $\C$ also gives rise to an adjunction
$$\langle V_{\C}, G_{\C}, \eta_{\C}, \varepsilon_{\C} \rangle: \ALG_{\C}
\rightharpoonup \ALG,$$
where
$$V_{\C} : \ALG_{\C} \rightarrow \ALG$$ is given on objects $\langle R,f\rangle$ in $\ALG_{\C}$
by $V_{\C}\langle R,f\rangle = R$.
The functor
$$G_{\C}:\ALG \rightarrow \ALG_{\C}$$ is defined on objects
$A$ in $\ALG$ by $G_{\C}(A) = \langle A^{\NN}, \delta_A\rangle$.
The natural transformations
$\eta_{\C}$ and $\varepsilon_{\C}$ are defined similarly to $\eta$ and $\varepsilon$ in Eqs.~(\mref{eq:eta1}) and (\mref{eq:vare1}), respectively.
Consequently there is a uniquely defined cocomparison functor
$H:\DIF \rightarrow \ALG_{\C}$ such that $H  G = G_{\C}$ and $V_{\C} H = V$.

\begin{coro}$($\cite[Corollary~3.5]{ZGK}$)$
The cocomparison functor $H:\DIF \rightarrow \ALG_{\C}$ is
an isomorphism, i.e., $\DIF$ is comonadic over $\ALG$.
\mlabel{co:comonadic}
\end{coro}

\subsection{Lifting comonads on $\RBA$}
\mlabel{ss:colift}
In view of the categorical study, we rephrase Proposition~\mref{prop:extrb} in terms of categories of various operated algebras. See~\mcite{Gop,Ku} for related studies.

Recall from Eq.~(\mref{eq:T}) that
$$\calt:=xy+\bfk[x]+y\bfk[x] = \{
xy-(\phi(x) +y\psi(x))\,|\, \phi,\psi\in\bfk[x]\}.
$$
Let $\OA$ denote the category of operated algebras and, for a given $\omega\in\calt$, let $\ODA_\omega$ denote the category of type $\omega$ operated differential algebras of weight $\lambda$ in Definition~\mref{de:oda}. Thanks to Proposition~\mref{prop:extrb}, we obtain a functor
\begin{equation}
 G_\omega: \OA \to \ODA_\omega
 \mlabel{eq:coefun}
\end{equation}
given on objects
$(R, Q)$ in $\OA$ by $G_\omega(R, Q) = (R^{\NN}, \partial_R, \lqt)$ and on morphisms
$\varphi:(R, Q) \rightarrow (S, P)$ in $\OA$ by $(G_\omega(\varphi))_n(f) := \varphi(f_n)$ for any $f\in R^{\NN}$ and $n\in \NN$.

\begin{defn}
For a given $\omega\in\calt$ and $\lambda\in \bfk$, we say that the triple $(R, d, P)$ is a
{\bf type $\omega$ differential Rota-Baxter algebra of weight $\lambda$} if
\begin{enumerate}
\item $(R,d)$ is a differential algebra of weight $\lambda$,
\item $(R,P)$ is a Rota-Baxter algebra of weight $\lambda$, and
\item $\omega(d, P)=0$, that is,
   \begin{equation}
   dP=\phi(d)+P\psi(d).
   \mlabel{eq:dP1}
   \end{equation}
\end{enumerate}
\mlabel{de:DRBt}
\end{defn}

Next we will provide an example of a type $\omega$ differential Rota-Baxter algebra of weight $0$.

\begin{exam}
{\rm
Recall a well-known fact from analysis: Let $f(x,y)$ be a function with two variables $(x,y)\in \RR\times[0,+\infty)$ where $\RR$ is the field of real numbers. If $f(x,y)$ and $\frac{\partial}{\partial x} f(x,y)$ are continuous, then
\begin{equation}
\frac{\partial}{\partial x}\int_0^y f(x,t)dt=\int_{0}^y \frac{\partial}{\partial x} f(x,t)dt.
\mlabel{eq:dPcomm}
\end{equation}
Let $A$ be an $\RR$-algebra consisting of all two-variable functions $f(x,y)$ defined on $\RR\times[0,\infty)$ with the property that both $f(x,y)$ and $\frac{\partial}{\partial x} f(x,y)$ are continuous, together with the function product:
$$(f\cdot g)(x,y)=f(x,y)g(x,y)\quad\text{for any}\ f, g\in A.$$
Define two $\RR$-linear operators $q$ and $Q$ on $A$ by
$$q(f)=\frac{\partial}{\partial x} f(x,y), \quad Q(f)=\int_0^y f(x,t)dt.$$
We can check that $q$ is a differential operator of weight $0$ and $Q$ is a Rota-Baxter operator of weight $0$. Eq.~(\mref{eq:dPcomm}) implies that $q$ and $Q$ commute, $qQ=Qq$. Thus $(A,q,Q)$ is a type $\omega$ differential Rota-Baxter algebra of weight $0$ where $\omega=xy-yx$.
}
\end{exam}

The category of type $\omega$ differential Rota-Baxter algebras of weight $\lambda$ will be denoted by $\DRB_\omega$. Note that $\DRB_\omega$ is a subcategory of $\ODA_\omega$.

We make the following assumption for a given $\omega\in\calt$ for the rest of Section~\mref{ss:colift}:

\begin{assumption}
{\rm
For every Rota-Baxter algebra $(R,P)$ of weight $\lambda$, the coextension $\lpt$ of $P$ in Proposition~\ref{prop:extrb} also gives a Rota-Baxter algebra $(R^\NN, \lpt)$ of weight $\lambda$.
}
\mlabel{as:coext}
\end{assumption}

This assumption amounts to the requirement that the functor $G_\omega: \OA \to \ODA_\omega$ in Eq.~(\mref{eq:coefun}) restricts to a functor
$$G_\omega:\RBA \rightarrow \DRB_\omega.$$
Let $V_\omega:\ODA_\omega \rightarrow \OA$ denote the forgetful
functor by forgetting the differential structure. Then $V_\omega$ restricts to a functor
$$V_\omega:\DRB_\omega \rightarrow \RBA.$$

\begin{lemma}
Let $\omega\in\calt$ satisfy Assumption~\mref{as:coext}, and $(R, d, P)\in\DRB_\omega$. Then we have
\begin{equation}
\lpt\eta_{(R,d)}=\eta_{(R,d)} P,
\mlabel{eq:etalemma}
\end{equation}
where $\eta_{(R,d)}$ is given in Eq.~$($\mref{eq:eta1}$)$.
\mlabel{lem:lpeta}
\end{lemma}
\begin{proof}
For all $n\in\NN$, we obtain $(\lpt\eta_{(R,d)})_n=\lpt_n\eta_{(R,d)}$ and $(\eta_{(R,d)} P)_n=d^n P$. Then
Eq.~(\mref{eq:etalemma}) is equivalent to
\begin{equation}
\lpt_n\eta_{(R,d)}=d^n P\quad\text{for all}\  n\in\NN.
\mlabel{eq:etacomlemma}
\end{equation}
We will prove by induction on $n\in\NN$ that Eq.~(\mref{eq:etacomlemma}) holds.

First for any $u\in R$,
\begin{equation*}
(\lpt_0\eta_{(R,d)})(u)=
P(\eta_{(R,d)}(u)_0)=P(u)=(d^0 P)(u).
\mlabel{eq:ass01}
\end{equation*}

Next assume that for a given $k\in \NN$,
\begin{equation}
\lpt_k\eta_{(R,d)}=d^k P
\mlabel{eq:assn1}
\end{equation}
holds. For every $v\in R$, we have
\begin{equation*}
((\partial_R^i)_\ell\eta_{(R,d)})(v)=d^{\ell+i}(v)=d^\ell(d^i(v))\quad \text{for all}\  i, \ell\in\NN.
\mlabel{eq:pari}
\end{equation*}
Then
\begin{equation}
\phi(\partial_R)_k\eta_{(R,d)}=d^k\phi(d), \quad \psi(\partial_R)\eta_{(R,d)}=\eta_{(R,d)}\psi(d).
\mlabel{eq:pari0}
\end{equation}
Since
\begin{eqnarray*}
(\lpt_{k+1}\eta_{(R,d)})(v)&=&\lpt_{k+1}(\eta_{(R,d)}(v))\\
&=&(\phi(\partial_R)_k+\lpt_k\psi(\partial_R))(\eta_{(R,d)}(v))
\quad \text{(by Eq.~(\ref{eq:lptn}))}\\
&=&(d^k\phi(d))(v)+
((\lpt_k\eta_{(R,d)})\psi(d))(v)\quad \text{(by Eq.~(\ref{eq:pari0}))}\\
&=&(d^k\phi(d))(v)+
((d^kP)\psi(d))(v)\quad \text{(by Eq.~(\ref{eq:assn1}))}
\end{eqnarray*}
and
\begin{eqnarray*}
(d^{k+1} P)(v)&=&(d^k(dP))(v)\\
&=&(d^k(\phi(d)+P\psi(d)))(v)\quad \text{(by Eq.~(\ref{eq:dP}))}\\
&=&(d^k\phi(d))(v)+
(d^k(P\psi(d)))(v),
\end{eqnarray*}
we obtain $(\lpt_{k+1}\eta_{(R,d)})(v)=(d^{k+1} P)(v).$
This completes the induction.
\end{proof}

\begin{prop}
Let $\omega\in\calt$ satisfy Assumption~\mref{as:coext}. The functor $G_\omega:\RBA \rightarrow \DRB_\omega$ is the right
adjoint of the forgetful functor $V_\omega:\DRB_\omega \rightarrow \RBA$. In other words, there is an adjunction $\langle V_\omega, G_\omega, \eta^\omega, \varepsilon^\omega \rangle:\DRB_\omega\rightharpoonup\RBA$.
\mlabel{pp:rightadjCt}
\end{prop}

\begin{proof}
By~\cite{Ma}, it is equivalent to show that there are two natural
transformations
$$\varepsilon^\omega: V_\omega G_\omega \rightarrow \id_{\RBA},\quad \eta^\omega: \id_{\DRB_\omega} \rightarrow G_\omega V_\omega$$
that
satisfy the equations
$$G_\omega\varepsilon^\omega  \circ\eta^\omega G_\omega = G_\omega, \quad \varepsilon^\omega V_\omega \circ V_\omega\eta^\omega = V_\omega.$$

For any $(A,P) \in \RBA$, define
\begin{equation*}
\varepsilon^\omega_{(A,P)}: (V_\omega G_\omega)(A,P) = (A^{\NN},\lpt) \rightarrow (A,P), \quad \varepsilon^\omega_{(A,P)}(f): = f_0 = \varepsilon_A(f) \quad \text{ for all }\ f \in A^{\NN}.
\mlabel{eq:vare2t}
\end{equation*}
By Proposition~\mref{prop:cofree}, $\varepsilon^\omega_{(A,P)}$ is an algebra homomorphism.
By
$$\varepsilon^\omega_{(A, P)}(\lpt(f))=P(f_0)=P(\varepsilon^\omega_{(A, P)}(f)),$$
we find that $\varepsilon^\omega_{(A ,P)}$ is a homomorphism of Rota-Baxter algebras.
If $\varphi:(A,P) \rightarrow (A',P')$ is any homomorphism of Rota-Baxter algebras, then
$$\varepsilon^\omega_{(A',P')}  (V_\omega G_\omega)(\varphi) = \varphi  \varepsilon^\omega_{(A,P)},$$
that is, $\varepsilon^\omega$
is a natural transformation.

For any $(R,d,P) \in \DRB_\omega$, define $\eta^\omega_{(R,d,P)}:(R,d,P) \rightarrow (R^{\NN}, \partial_R,\lpt)$ by
\begin{equation*}
(\eta^\omega_{(R,d,P)}(x))_n := d^{n}(x) = (\eta_{(R,d)}(x))_n
\mlabel{eq:eta3t}
\end{equation*}
for all $x \in R$ and $n \in \NN$.
Then $\eta^\omega_{(R,d,P)}$ is an algebra homomorphism.
Also we obtain
$$ \partial_R  \eta^\omega_{(R,d,P)} = \eta^\omega_{(R,d,P)}  d $$
and $\lpt\eta^\omega_{(R,d,P)}=\eta^\omega_{(R,d,P)} P$ follows from Lemma~\mref{lem:lpeta} and $\eta^\omega_{(R,d,P)}=\eta_{(R,d)}$.
Then $\eta^\omega_{(R,d,P)}$ is a morphism in $\DRB_\omega$.
Furthermore, if $\varphi:(R, d, P) \rightarrow (R',d',P')$ is a morphism in $\DRB_\omega$,
then one sees that $\eta^\omega_{(R',d',P')}  \varphi = (G_\omega V_\omega)(\varphi)  \eta^\omega_{(R,d,P)}.$
Hence $\eta^\omega$ is a natural transformation.

To check that $G_\omega\varepsilon^\omega \circ \eta^\omega G_\omega = G_\omega$, let $(A,P)\in \RBA$,
$f \in A^{\NN}$ and $n \in \NN$.  Then
\begin{eqnarray*}
(G_\omega\varepsilon^\omega_{(A,P)}(\eta^\omega_{(A^{\NN},\partial_A,\lpt)}(f)))_n &=&
\varepsilon^\omega_{(A,P)}(\eta^\omega_{(A^{\NN},\partial_A,\lpt)}(f)_n) =
\varepsilon^\omega_{(A,P)}(\partial_{A}^{n}(f))
=(\partial_{A}^{n}(f))_0 =
f_{0+n} = f_n.
\end{eqnarray*}
Similarly, to check that $\varepsilon^\omega V_\omega \circ V_\omega\eta^\omega = V_\omega$,
let $(R,d,P) \in\DRB_\omega$ and $x \in R$.  Then
$$\varepsilon^\omega_{(R,P)}(\eta^\omega_{(R,d,P)}(x)) =
(\eta^\omega_{(R,d,P)}(x))_0 = d^{0}(x) = x,$$
as desired.
\end{proof}

The adjunction $\langle V_\omega, G_\omega, \eta^\omega, \varepsilon^\omega \rangle:\DRB_\omega\rightharpoonup\RBA$ gives a comonad
$\C_\omega = \langle C_\omega,\varepsilon^\omega,\delta^\omega \rangle$
 on the category
$\RBA$, where
$C_\omega:\RBA \rightarrow \RBA$
is the functor
whose value for any $(R,P) \in \RBA$ is $C_\omega(R,P) = (R^{\NN}, \lpt)$ and $\delta^\omega:C_\omega\rar C_\omega C_\omega$ is defined by $\delta^\omega := V_\omega\eta^\omega G_\omega$.
In other words, for any $(R,P) \in \RBA$,
$$\delta^\omega_{(R,P)}:(R^\NN,\lpt) \rightarrow ((R^{\NN})^{\NN}, \widehat{\lpt}^\omega),\quad \delta^\omega_{(R,P)}(f)=\delta_R(f)\quad\text{for all}\  f \in R^{\NN}.$$
Consequently, the comonad $\C_\omega$ is a lifting of $\C$ on $\RBA$.

Similarly, there is a category
of $\C_\omega$-coalgebras, denoted by $\RBA_{\C_\omega}$.
The comonad $\C_\omega$ also induces an adjunction
\begin{equation*}
\langle (V_\omega)_{\C_\omega}, (G_\omega)_{\C_\omega}, \eta^\omega_{\C_\omega}, \varepsilon^\omega_{\C_\omega} \rangle: \RBA_{\C_\omega}
\rightharpoonup \RBA,
\mlabel{eq:adjc'}
\end{equation*}
where
$(V_\omega)_{\C_\omega} : \RBA_{\C_\omega} \rightarrow \RBA$ is given on objects $\langle (A,P),f\rangle$ in $\RBA_{\C_\omega}$
by $$(V_\omega)_{\C_\omega}\langle (A,P),f\rangle = (A,P)$$ and on morphisms
$\varphi:\langle (A,P),f\rangle\rightarrow\langle (B,Q),g\rangle$ in $\RBA_{\C_\omega}$ by $(V_\omega)_{\C_\omega}(\varphi)=\varphi$.  The functor
$(G_\omega)_{\C_\omega}:\RBA \rightarrow \RBA_{\C_\omega}$ is defined on objects
$(A,P)$ in $\RBA$ by $$(G_\omega)_{\C_\omega}(A,P) = \langle(A^{\NN}, \lpt),\delta^\omega_{(A,P)}\rangle$$ and on morphisms
$\phi:(A,P)\rightarrow (B,Q)$ in $\RBA$ by $(G_\omega)_{\C_\omega}(\phi)= \phi^{\NN}$.  The natural transformations $\varepsilon^\omega_{\C_\omega}$ and
$\eta^\omega_{\C_\omega}$ are defined similarly to $\varepsilon^\omega$ and $\eta^\omega$, respectively.

Then there is a uniquely defined cocomparison functor
$H_\omega:\DRB_\omega \rightarrow \RBA_{\C_\omega}$ such that $H_\omega G_\omega = (G_\omega)_{\C_\omega}$ and $(V_\omega)_{\C_\omega} H_\omega = V_\omega$.

Applying the dual of \cite[Proposition 2.5(b)]{ZGK}, we obtain
\begin{coro}
Suppose that $\omega\in\calt$ satisfy Assumption~\mref{as:coext}. Then the cocomparison functor $H_\omega:\DRB_\omega \rightarrow \RBA_{\C_\omega}$ is
an isomorphism, i.e., $\DRB_\omega$ is comonadic over $\RBA$.
\mlabel{co:Rcomonadic}
\end{coro}

\subsection{Lifting monads on $\DIF$}
\mlabel{ss:lift}
We let $\OA_0$ denote the category of operated algebras $(R, q)$ with the property $q(\bfone_R)=0$, and let $\ORB_\omega$ denote the category of type $\omega$ operated Rota-Baxter algebras of weight $\lambda$.
Thanks to Proposition~\mref{prop:extd}, we obtain a functor
\begin{equation}
 F_\omega: \OA_0 \to \ORB_\omega
 \mlabel{eq:extfun}
\end{equation}
given on
objects $(R,q)$ in $\OA_0$ by $F_\omega(R,q) = (\sha(R),\lqst,P_R)$ and on morphisms
$\varphi:(R,q) \rar (S,d)$ in $\OA_0$ by $$F_\omega(\varphi)(\sum^{k}_{i=1}a_{i0}\otimes a_{i1}\otimes\cdots \otimes a_{in_{i}})=\sum^{k}_{i=1}\varphi(a_{i0})\otimes \varphi(a_{i1})\otimes\cdots \otimes \varphi(a_{in_{i}}),\quad
\sum^{k}_{i=1}a_{i0}\otimes a_{i1}\otimes\cdots \otimes a_{in_{i}}\in\sha(R).$$

The category $\DRB_\omega$, which is a subcategory of $\ODA_\omega$, is also a subcategory of $\ORB_\omega$.
We consider the following condition for a given $\omega\in\calt$ in Section~\mref{ss:lift}:
\begin{assumption}
{\rm
For every differential algebra $(R,d)$ of weight $\lambda$, the extension $\ldt$ of $d$ to $\sha(R)$ in Proposition~\ref{prop:extd} is a differential operator of weight $\lambda$.
\mlabel{as:ext}
}
\end{assumption}
This condition amounts to assuming that the functor $F_\omega: \OA_0 \to \ORB_\omega$ in Eq.~(\mref{eq:extfun}) restricts to a functor
$$F_\omega:\DIF \rar \DRB_\omega.$$

Let $U_\omega:\ORB_\omega \rightarrow \OA_0$ be the forgetful
functor by forgetting the Rota-Baxter algebra structure. Then $U_\omega$ restricts to a functor
$$U_\omega:\DRB_\omega \rar \DIF.$$

\begin{lemma}
Let $\omega\in\calt$ satisfy Assumption~\mref{as:ext} and $(R, d, P)\in\DRB_\omega$. Then we have
\begin{equation*}
\varepsilon_{(R, P)}\ldt=d\varepsilon_{(R, P)},
\mlabel{eq:varelemma}
\end{equation*}
where $\varepsilon_{(R, P)}$ is given in Eq.~$($\mref{eq:vare0}$)$.
\mlabel{lem:edcomm}
\end{lemma}
\begin{proof}
We will prove by induction on $n\in\NN_+$ that
$$\varepsilon_{(R, P)}\ldt|_{\oplus_{i=1}^nR^{\ot i}}=d\varepsilon_{(R, P)}|_{\oplus_{i=1}^nR^{\ot i}}$$
holds.

First when $n=1$,
\begin{equation*}
\varepsilon_{(R, P)}\ldt|_R=d=d\varepsilon_{(R, P)}|_R.
\mlabel{eq:ass02}
\end{equation*}

Next we assume that for a given $k\in \NN_+$,
\begin{equation*}
\varepsilon_{(R, P)}\ldt|_{\oplus_{i=1}^kR^{\ot i}}=d\varepsilon_{(R, P)}|_{\oplus_{i=1}^kR^{\ot i}}
\mlabel{eq:assn2}
\end{equation*}
holds. Together with Eq.~(\ref{eq:qcon}), we obtain
\begin{equation}
\varepsilon_{(R, P)}\phi(\ldt)|_{\oplus_{i=1}^kR^{\ot i}}=\phi(d)\varepsilon_{(R, P)}|_{\oplus_{i=1}^kR^{\ot i}},\quad
\varepsilon_{(R, P)}\psi(\ldt)|_{\oplus_{i=1}^kR^{\ot i}}=\psi(d)\varepsilon_{(R, P)}|_{\oplus_{i=1}^kR^{\ot i}}.
\mlabel{eq:assn21}
\end{equation}
Let $v=v_0\ot v'\in R^{\ot (k+1)}$ with $v'\in R^{\ot k}$.
Since
\begin{eqnarray*}
&&\varepsilon_{(R, P)}(\ldt(v))\\
&=&\varepsilon_{(R, P)}(d(v_0)P_R(v')+(v_0+\lambda d(v_0))(\phi(\ldt)+P_R\psi(\ldt))(v'))\quad \text{(by Eq.~(\ref{eq:1leibdform}))}\\
&=&d(v_0)P(\varepsilon_{(R, P)}(v'))+(v_0+\lambda d(v_0))(\varepsilon_{(R, P)}\phi(\ldt))(v')\\
&&+(v_0+\lambda d(v_0))(P(\varepsilon_{(R, P)}\psi(\ldt)))(v')\quad \text{(by Eq.~(\ref{eq:vare0}))}\\
&=&d(v_0)P(\varepsilon_{(R, P)}(v'))+(v_0+\lambda d(v_0))(\phi(d)\varepsilon_{(R, P)})(v')\\
&&+(v_0+\lambda d(v_0))(P(\psi(d)\varepsilon_{(R, P)}))(v')\quad \text{(by Eq.~(\ref{eq:assn21}))}
\end{eqnarray*}
and
\begin{eqnarray*}
&&d(\varepsilon_{(R, P)}(v))\\
&=&d(v_0P(\varepsilon_{(R, P)}(v')))\quad \text{(by Eq.~(\ref{eq:vare0}))}\\
&=&d(v_0)P(\varepsilon_{(R, P)}(v'))+(v_0+\lambda d(v_0))(d P)(\varepsilon_{(R, P)}(v'))\quad \text{(by Eq.~(\ref{eq:der1})) }\\
&=&d(v_0)P(\varepsilon_{(R, P)}(v'))+(v_0+\lambda d(v_0))(\phi(d)+P\psi(d))(\varepsilon_{(R, P)}(v'))\quad \text{(by Eq.~(\ref{eq:dP}))}\\
&=&d(v_0)P(\varepsilon_{(R, P)}(v'))+(v_0+\lambda d(v_0))(\phi(d)\varepsilon_{(R, P)})(v')+(v_0+\lambda d(v_0))(P(\psi(d)\varepsilon_{(R, P)}))(v'),
\end{eqnarray*}
we obtain $(\varepsilon_{(R, P)}\ldt)(v)=(d\varepsilon_{(R, P)})(v)$. This completes the induction.
\end{proof}

\begin{prop}
Let $\omega \in\calt$ satisfy Assumption~\mref{as:ext}. The functor $F_\omega:\DIF \rightarrow \DRB_\omega$ is the left adjoint of the forgetful functor
$U_\omega:\DRB_\omega \rightarrow \DIF$. In other words, there is an adjunction $\langle F_\omega, U_\omega, \eta^\omega, \varepsilon^\omega \rangle:\DIF\rightharpoonup \DRB_\omega$.
\mlabel{pp:leftadjCt}
\end{prop}
\begin{proof}
By~\cite{Ma}, we just need to define natural
transformations
$$\eta^\omega: \id_{\DIF} \rightarrow U_\omega F_\omega,\quad \varepsilon^\omega: F_\omega U_\omega \rightarrow \id_{\DRB_\omega}$$
that satisfy the equations
$U_\omega\varepsilon^\omega \circ \eta^\omega U_\omega = U_\omega$ and $\varepsilon^\omega F_\omega \circ F_\omega\eta^\omega = F_\omega$.

For any $(R, d) \in \DIF$, we define
\begin{equation*}
\eta^\omega_{(R, d)}:(R, d)\rightarrow (\sha(R),\ldt)
\mlabel{eq:eta2t}
\end{equation*}
to be the natural embedding map. Then $\eta^\omega_{(R, d)}$ is an algebra homomorphism.
Since
$$(\eta^\omega_{(R, d)} d)(r)=d(r)=(\ldt\eta^\omega_{(R, d)})(r)\quad \text{for all}\  r\in R,$$
$\eta^\omega_{(R, d)}$ is a morphism in $\DIF$.
As $\eta^\omega_{(R, d)}$ is a natural embedding, $\eta^\omega$ is a natural transformation.

For any $(R,d,P) \in \DRB_\omega$, define $\varepsilon^\omega_{(R,d,P)}: (\sha(R),\ldt,P_R) \rightarrow (R,d,P)$ by
\begin{equation*}
\varepsilon^\omega_{(R,d,P)}\Big(\sum^{k}_{i=1}a_{i0}\otimes a_{i1}\otimes\cdots \otimes a_{in_{i}}\Big) :=
\sum^{k}_{i=1}a_{i0}P(a_{i1}P(\cdots P(a_{in_{i}})\cdots))
\mlabel{eq:vare3t}
\end{equation*}
for any $\sum\limits^{k}_{i=1}a_{i0}\otimes a_{i1}\otimes\cdots \otimes a_{in_{i}} \in \sha(R).$
One can check that $\varepsilon^\omega_{(R,d,P)}$ is a Rota-Baxter algebra homomorphism. Further,
$\varepsilon^\omega_{(R, d, P)}\ldt=d\varepsilon^\omega_{(R, d, P)}$ by Lemma~\mref{lem:edcomm} and $\varepsilon^\omega_{(R, d, P)}=\varepsilon_{(R, P)}$.
Therefore,
$\varepsilon^\omega_{(R, d, P)}$ is a morphism in $\DRB_\omega$.
If
$\varphi:(R,d,P) \rightarrow (R',d',P')$ is any morphism in $\DRB_\omega$, then $\varepsilon^\omega_{(R',d',P')}  (F_\omega U_\omega)(\varphi) = \varphi \varepsilon^\omega_{(R,d,P)}$, that is, $\varepsilon^\omega$
is a natural transformation.

To see that $U_\omega\varepsilon^\omega \circ \eta^\omega U_\omega= U_\omega$, let $(R, D, P)\in \DRB_\omega$ and $r \in R$.  Then
$$\varepsilon^\omega_{(R, D, P)}(\eta^\omega_{(R, D)}(r))=\varepsilon^\omega_{(R, D, P)}(r)=r.$$
Similarly, to see that $\varepsilon^\omega F_\omega\circ F_\omega\eta^\omega = F_\omega$,
let $(A, d) \in \DIF$ and $a_0\ot a_1\ot\cdots\ot a_n \in \sha(A)$.  Then
$$\varepsilon^\omega_{(\ssha(A),\ldt,P_A)}((F_\omega\eta^\omega_{(A, d)})(a_0\ot a_1\ot\cdots\ot a_n))=\varepsilon^\omega_{(\ssha(A),\ldt,P_A)}
(a_0\otimes a_1\otimes\cdots\otimes a_n)=a_0\ot a_1\ot\cdots\ot a_n,$$
as needed.
\end{proof}

The adjunction $\langle F_\omega, U_\omega, \eta^\omega, \varepsilon^\omega \rangle:\DIF\rightharpoonup \DRB_\omega$ gives rise to a monad $\T_\omega=\langle T_\omega,\eta^\omega,\mu^\omega \rangle$ on $\DIF$,  where
$T_\omega = U_\omega F_\omega:\DIF \rightarrow \DIF$
is a functor and $\mu^\omega:=U_\omega\varepsilon^\omega F_\omega: T_\omega T_\omega\rightarrow T_\omega$. Note that the monad $\T_\omega$ is a lifting of $\T$ on $\DIF$.

As before, the monad $\T_\omega$ induces a category
of $\T_\omega$-algebras, denoted by $\DIF^{\T_\omega}$, and gives rise to an adjunction
\begin{equation*}
\langle F_\omega^{\T_\omega}, U_\omega^{\T_\omega},(\eta^\omega)^{\T_\omega} , (\varepsilon^\omega)^{\T_\omega} \rangle:\DIF
\rightharpoonup \DIF^{\T_\omega},
\mlabel{eq:adjt'}
\end{equation*}
where
$F_\omega^{\T_\omega} : \DIF \rightarrow \DIF^{\T_\omega}$ is given on objects $(A, d)$ in $\DIF$
by
$$F_\omega^{\T_\omega}(A, d) = \langle(\sha(A),\ldt),\mu^\omega_{(A, d)}\rangle$$ and on morphisms
$\varphi:(A, d)\rightarrow (A', d')$ in $\DIF$ by $F_\omega^{\T_\omega}(\varphi)=F_\omega(\varphi)$, and
$U_\omega^{\T_\omega}:\DIF^{\T_\omega}\rightarrow \DIF$
is defined on objects
$\langle (A, d),f\rangle$ in $\DIF^{\T_\omega}$ by $$U_\omega^{\T_\omega}\langle (A, d),f\rangle = (A, d)$$ and on morphisms
$\varphi:\langle (A, d),f\rangle\rightarrow\langle (A',d'),f'\rangle$ in $\DIF^{\T_\omega}$
by $U_\omega^{\T_\omega}(\varphi)=\varphi$.  The natural transformations $(\eta^\omega)^{\T_\omega}$ and $(\varepsilon^\omega)^{\T_\omega}$ are defined similarly as $\eta^\omega$ and $\varepsilon^\omega$, respectively.
Then there is a uniquely defined comparison functor
$K_\omega:\DRB_\omega \rightarrow \DIF^{\T_\omega}$ such that $K_\omega F_\omega=F_\omega^{\T_\omega}$
and $U_\omega^{\T_\omega} K_\omega=U_\omega$.

As a special case of~\cite[Proposition 2.5(b)]{ZGK}, we obtain
\begin{coro}
Suppose that $\omega\in\calt$ satisfy Assumption~\mref{as:ext}. Then the comparison functor $K_\omega:\DRB_\omega \rightarrow \DIF^{\T_\omega}$ is
an isomorphism, that is, $\DRB_\omega$ is monadic over $\DIF$.
\mlabel{co:drbmon}
\end{coro}

\subsection{The main result}
\mlabel{sec:mix}

We now give the equivalence among the existence of (co)extensions in Section~\mref{sec:ext}, of liftings of (co)monads in Section~\mref{ss:colift} and Section~\mref{ss:lift}, and mixed distributive laws involving type $\omega$ differential Rota-Baxter algebras.

We first recall from~\cite{HW} some background information on mixed distributive laws, a generalization of the notion of a distributive law introduced by J. Beck in his fundamental work~\cite{Be}.
\begin{defn}
Let a category $\A$, a monad ${\mathbf{T}} = \langle T, \eta, \mu \rangle$ and a comonad $\C = \langle C, \varepsilon, \delta\rangle$ on $\A$ be given.
A {\bf mixed distributive law of $\T$ over $\C$} is a natural transformation $\beta: TC \rightarrow CT$ satisfying the following conditions.

\begin{enumerate}
\item $\beta \circ \eta C = C\eta$;
\item $\varepsilon T \circ \beta = T\varepsilon$;
\item $\delta T \circ \beta = C\beta \circ \beta C \circ T\delta$;
\item $\beta \circ \mu C = C\mu \circ \beta T \circ T\beta$.
\end{enumerate}
\mlabel{def:mdl}
\end{defn}

\begin{theorem}
The following statements are equivalent.
\begin{enumerate}
\item
There exists a mixed distributive law of $\T$ over $\C$.
\item
There exists a comonad
$\tilde{\C} = \langle \tilde{C}, \tilde{\varepsilon},\tilde{\delta} \rangle$ on $\A^{\T}$ which lifts $\C$ $($i.e., $U^{\T} \tilde{C}=C U^{\T}$, $U^{\T}\tilde{\varepsilon}=\varepsilon U^{\T}$, $U^{\T}\tilde{\delta}=\delta U^{\T}$$)$.

\item There exists a monad $\tilde{\T}=\langle\tilde{T}, \tilde{\eta}, \tilde{\mu}\rangle$ on $\A_{\C}$ which lifts $\T$.
\end{enumerate}
\mlabel{theorem:3equ}
\end{theorem}
\begin{proof}
Theorem~2.2 and Remark~2.3 from~\cite{Wo2} give the correspondence among lifting $\mathcal V$-comonads, lifting $\mathcal V$-monads and $\mathcal V$-mixed distributive laws, where $\mathcal V$ is a symmetric monoidal closed category. The theorem follows as a special case when $\mathcal V$ is taken to be the category of sets.
\end{proof}

Recall $\calt=\left\{
xy-(\phi(x) +y\psi(x)) \,|\, \phi, \psi\in\bfk[x]
\right\}$ as in Eq.~(\mref{eq:T}). Now we have arrived at the main result of this paper.

\begin{theorem}
Let $\omega\in\calt$ be given. The following statements are equivalent:
\begin{enumerate}
\item
For every Rota-Baxter operator $P$ on every algebra $R$, the unique coextension $\lpt$ of $P$ to $R^\NN$ given in Proposition~\mref{prop:extrb} is a Rota-Baxter operator.
\mlabel{it:extrb}
\item
For every differential operator $d$ on every algebra $R$, the unique extension $\ldt$ of $d$ to $\sha(R)$ given in Proposition~\mref{prop:extd} is a differential operator.
\mlabel{it:extd}
\item
The functor $G_\omega:\OA\to \ODA_\omega$ in Eq.~(\mref{eq:coefun}) restricts to a functor $G_\omega:\RBA\to \DRB_\omega$.
\mlabel{it:coefun}
\item
The functor $F_\omega:\OA_0\to \ORB_\omega$ in Eq.~(\mref{eq:extfun}) restricts to a functor $F_\omega:\DIF\to \DRB_\omega$.
\mlabel{it:extfun}
\item
There is a lifting $\tilde{\C}= \langle \tilde{C},\tilde{\varepsilon},\tilde{\delta} \rangle$ of $\mathbf{C}$ on $\RBA$ where $\tilde{C}(R,P):=(R^{\NN}, \widetilde{P})$, and an isomorphism $\widetilde{H}:\DRB_\omega\rightarrow\RBA_{\tilde{\C}}$ over $\RBA$ given by
$$\widetilde{H}(R, d, P)=\langle (R, P), \theta_{(R, d, P)}:(R, P)\rightarrow (R^{\NN}, \widetilde{P})\rangle.$$ Here $\theta_{(R, d, P)}(u)_{n}:=d^n(u)$ for all $u\in R, n\in\NN$.
\mlabel{it:liftcomo}
\item
There is a lifting $\tilde{\T}= \langle \tilde{T},\tilde{\eta},\tilde{\mu} \rangle$ of $\mathbf{T}$ on $\DIF$ where $\tilde{T}(R,d):=(\sha(R),\tilde{d})$, and an isomorphism $\widetilde{K}:\DRB_\omega\rightarrow\DIF^{\tilde{\T}}$ over $\DIF$ given by $$\widetilde{K}(R, d, P)=\langle (R, d), \vartheta_{(R, d, P)}:(\sha(R),\tilde{d})\rightarrow (R, d)\rangle .$$
Here $\vartheta_{(R, d, P)}(v_0\ot v_1\ot\cdots \ot v_m):=v_0P(v_1P(\cdots P(v_m)\cdots))$ for all $v_0\ot v_1\ot\cdots \ot v_m\in \sha(R)$.
\mlabel{it:liftmo}
\vspace{-0.15in}
\item
There is a mixed distributive law $\beta:TC\rightarrow CT$ such that $(\ALG_{\C})^{\tilde{\T}_\beta}$ is isomorphic to the category $\DRB_\omega$, where $\tilde{\T}_\beta$ is a lifting monad of $\T$ given by the mixed distributive law $\beta$.
\mlabel{it:mdl}
\end{enumerate}
\mlabel{thm:main}
\end{theorem}
We will use an algebraic description in a subsequent paper to give a classification of those constraint $\omega\in \Omega$ that satisfy these equivalent conditions in Theorem~\mref{thm:main}. More precisely, for a given $\omega\in \Omega$, we will tell whether the corresponding coextension of a Rota-Baxter operator $P$ on an algebra $R$ to $(R^\NN,\partial_R)$ is a Rota-Baxter operator again.

\begin{proof}
{(\mref{it:extrb}) $\Longrightarrow$ (\mref{it:coefun})}. Item~(\mref{it:extrb}) combined with Proposition~\mref{prop:extrb} gives $(R^\NN, \partial_R, \lpt)\in\DRB_\omega$.
\smallskip

\noindent
(\mref{it:extd}) $\Longrightarrow$ (\mref{it:extfun}).
Similarly, Item~(\mref{it:extd}) combined with Proposition~\mref{prop:extd} gives $(\sha(R), \ldt, P_R)\in\DRB_\omega$.
\smallskip

\noindent
(\mref{it:coefun}) $\Longrightarrow$ (\mref{it:liftcomo}).
In Proposition~\mref{pp:rightadjCt}, we obtain an adjunction $\langle V_\omega, G_\omega, \eta^\omega, \varepsilon^\omega \rangle:\DRB_\omega\rightharpoonup\RBA$
which gives a lifting comonad $\mathbf{C}_\omega$ of $\mathbf{C}$ on $\RBA$. Also there is an isomorphism $H_\omega:\DRB_\omega \rightarrow \RBA_{\C_\omega}$ in Corollary~\mref{co:Rcomonadic}, as required.
\smallskip

\noindent
(\mref{it:extfun}) $\Longrightarrow$ (\mref{it:liftmo}).
In Proposition~\mref{pp:leftadjCt}, we obtain an adjunction $\langle F_\omega, U_\omega, \eta^\omega, \varepsilon^\omega \rangle:\DIF\rightharpoonup \DRB_\omega$
which gives a lifting monad $\T_\omega$ of $\T$ on $\DIF$. Further, there is an isomorphism $K_\omega:\DRB_\omega \rightarrow \DIF^{\T_\omega}$ in Corollary~\mref{co:drbmon}, as required.
\smallskip

\noindent
(\mref{it:liftcomo}) $\Longrightarrow$ (\mref{it:extrb}).
The lifting comonad $\tilde{\C} = \langle \tilde{C},\tilde{\varepsilon},\tilde{\delta} \rangle$ on $\RBA$ induces an adjunction
$$\langle V_{\tilde{\C}}, G_{\tilde{\C}}, \eta_{\tilde{\C}}, \varepsilon_{\tilde{\C}} \rangle: \RBA_{\tilde{\C}}
\rightharpoonup \RBA.$$
Since $\tilde{\varepsilon}_{(R,P)}:(R^{\NN}, \widetilde{P})\rightarrow (R,P)$ is a morphism in $\RBA$, $\tilde{\varepsilon}_{(R,P)}(\widetilde{P}(f))=
P(\tilde{\varepsilon}_{(R,P)}(f))$ for all $f\in R^\NN$. That is, $\widetilde{P}_0(f)=P(f_0)$. Then $\widetilde{P}$ is a coextension of $P$.
Recall that
$$G_{\tilde{\C}}(R,P) =\langle (R^{\NN}, \widetilde{P}),\tilde{\delta}_{(R,P)}\rangle\in \RBA_{\tilde{\C}},\quad\widetilde{H}^{-1}\langle (R^{\NN}, \widetilde{P}),\tilde{\delta}_{(R,P)}\rangle=(R^{\NN}, \partial_R, \widetilde{P})\in \DRB_\omega.$$
Then $\partial_R\widetilde{P}=\phi(\partial_R)+\widetilde{P}\psi(\partial_R)$. By the uniqueness of the coextension in Proposition~\mref{prop:extrb}, we have $\widetilde{P}=\lpt$. Therefore, $\lpt$ is a Rota-Baxter operator.
\smallskip

\noindent
(\mref{it:liftmo}) $\Longrightarrow$ (\mref{it:extd}).
The lifting monad $\tilde{\T} = \langle \tilde{T},\tilde{\eta},\tilde{\mu} \rangle$ on $\DIF$ induces an adjunction
$$\langle F^{\tilde{\T}}, U^{\tilde{\T}},\eta^{\tilde{\T}} , \varepsilon^{\tilde{\T}} \rangle:\DIF
\rightharpoonup \DIF^{\tilde{\T}}.$$
Since $\tilde{\eta}_{(R,d)}:(R,d)\rightarrow (\sha(R), \tilde{d})$ is a morphism in $\DIF$, we have $\tilde{d}\tilde{\eta}_{(R,d)}=\tilde{\eta}_{(R,d)} d$. That is, $\tilde{d}(r)=d(r)$ for all $r\in R$. Then $\tilde{d}$ is an extension of $d$.
Recall that
$$F^{\tilde{\T}}(R,d) =\langle (\sha(R),\tilde{d}),\tilde{\mu}_{(R,d)}\rangle\in \DIF^{\tilde{\T}},\quad\widetilde{K}^{-1}\langle (\sha(R),\tilde{d}),\tilde{\mu}_{(R,d)}\rangle=(\sha(R), \tilde{d}, P_R)\in \DRB_\omega.$$ Let $u=u_0\ot u'\in R^{\ot (n+1)}$ with $u'\in R^{\ot n}$. Since $\tilde{d}$ is a differential operator of weight $\lambda$ on $\sha(R)$, we have
$$\tilde{d}(u)=\tilde{d}(u_0P_R(u'))=\tilde{d}(u_0)P_R(u')+(u_0+\lambda \tilde{d}(u_0))(\tilde{d}P_R)(u')\quad \text{(by Eq.~(\ref{eq:der1})) }.$$
Also by $\tilde{d} P_R=\phi(\tilde{d})+P_R\psi(\tilde{d})$, we have $\tilde{d}(u)=d(u_0)\ot u'+(u_0+\lambda d(u_0))(\phi(\tilde{d})+P_R\psi(\tilde{d}))(u').$ Further, we obtain $\tilde{d}(\oplus_{i=1}^{n}R^{\ot i})\subseteq\oplus_{i=1}^{n}R^{\ot i}$ which is proved by induction on $n\in\NN_+$. That is, $\tilde{d}$ satisfies Eqs.~(\mref{eq:1leibdform}) and (\mref{eq:qcon}).
Then by the uniqueness of the extension in Proposition~\mref{prop:extd}, we have $\tilde{d}=\ldt$. Therefore, $\ldt$ is a differential operator.
\smallskip

\noindent
(\mref{it:liftcomo})$\Longleftrightarrow$(\mref{it:liftmo})
$\Longleftrightarrow$(\mref{it:mdl}).
This follows directly from Theorem~\mref{theorem:3equ}, \cite[Theorem~2.4]{Wo2} and \cite[Theorem~5.3]{ZGK}.
\end{proof}

\smallskip

\noindent
{\bf Acknowledgements}:
This work is supported by the National Natural Science Foundation of China (Grant No. 11371178) and the China Scholarship Council (Grant No. 201606180084). Shilong Zhang thanks Rutgers University -- Newark for its hospitality during his visit August 2016 - August 2017.

\end{document}